\documentclass[12pt,twoside,openany]{article}
\usepackage{enumerate,amsmath,graphics,amssymb,graphicx,amscd,xypic,amscd,amsbsy,multirow,float,booktabs,verbatim}
\usepackage[round]{natbib}
\newcommand{\QED}{\hspace*{\fill}\rule{2.5mm}{2.5mm}}

\newtheorem{definition}{Definition}[section]

\newenvironment{proof}{\noindent{\bf Proof\ }}{\QED\\}

\newcommand{\mc}{\mathcal}
\newcommand{\mf}{\mathcal}
\newtheorem{proposition}{Proposition}[section]
\newtheorem{lemma}{Lemma}[section]

\newtheorem{corollary}{Corollary}[section]
\newtheorem{example}{Example}[section]

\begin{document}

\begin{center}

\vspace{0.5cm} {\large \bf ``Conditional information and definition of neighbor in categorical random fields"}\\
\vspace{1cm} Reza Hosseini, University of British Columbia,\\
333-6356 Agricultural Road,
Vancouver, BC, Canada, V6T1Z2\\
reza1317@gmail.com
\end{center}

\begin{abstract} We show that the definition of neighbor in
Markov random fields as defined by Besag (1974) when the joint
distribution of the sites is not positive is not well-defined. In
a random field with finite number of sites we study the conditions under which
giving the value at extra sites will change the belief
of an agent about one site. Also the conditions under which
the information from some sites is equivalent to giving the value
at all other sites is studied. These concepts provide an alternative to the
concept of neighbor for general case where the positivity condition of the joint does not hold.\\
\vspace{0.25cm}


\end{abstract}

\section{Introduction}

This paper studies the conditional probabilities and the definition of
neighbor in categorical random fields. These can be used to
describe spatial processes e.g. in plant ecology. We start by the
common definition of neighbor in Markov random fields and show
that the definition is not well-defined when the joint
distribution is not positive. Then we provide a framework to study
the conditional probabilities given various amount of
``information". For example, the conditional probability of one
site given some others. Since the usual definition of neighbor is
not well-defined when the ``positivity" condition of the joint
distribution does not hold, we introduce some new concepts of
``uninformative set", ``sufficient information set" and ``minimal
information set".

Suppose we have a finite random field consisting of $n$ sites.
The belief of an agent about one site  can be summarized by a
probability distribution and can be changed to a conditional distribution
by relieving new information which can be the value at some other
sites. We study when the new information changes the agent's belief
and what is ``sufficient'' information for the agent in the sense
that giving the information would be equivalent to giving the value
of all other sites. We answer some interesting questions along the
way. For example suppose agent 1 has less information than agent 2
regrading an event A and a new information is released. Now, suppose
that agent 1 does not change his belief about A. One might
conjecture that since agent 2 has more information, he as well will
not change his belief after receiving the new information. We show
this conjecture is wrong by counterexamples.

\section{Neighbor in categorical random fields}
 Suppose $(\Omega, \Sigma, P)$ is a probability space and
$\{X_i\}_{i=1}^n$ is a stochastic process. Each $X_i$ takes values
in $M_i,\;|M_i|=m_i< \infty$, and $P(x_i)>0,\;\forall x_i \in M_i$.
We use the shorthand notation:

\[P(x_i|x_{i_1}\cdots,x_{i_k})=P(X_i=x_i|X_{i_1}=x_{i_1}\cdots,X_{i_k}=x_{i_k}).\]

 \cite{besag} and  \cite{cressie},
defined the neighbor as follows:

\begin{definition} For site $i,\;i=1,\cdots,n$, site $j\neq i$ is
called a neighbor if and only if the functional form of the
$P(x_i|x_1,\cdots,x_{i-1},x_{i+1},\cdots,x_n)$ is dependent on
$x_j$.
\end{definition}

Note that in the above definition, we need to make sure that the
conditional probability is defined. The above conditional
probability is defined on
\[E_i=\{(x_1,\cdots,x_n)\;
|\;P(x_1,\cdots,x_{i-1},x_{i+1},\cdots,x_n)>0\}.\]

We show in the following example this definition is not
well-defined in general since the functional form is not unique.

\begin{example}
Let $U_1,\cdots,U_4$ denote a random sample from the uniform
distribution that take only values 0 and 1 each with probability
1/2. Define:

\begin{align*}
X_1=U_1+U_2,\\
X_2=[X_1]+U_3,\\
X_3=[X_2]+U_4,\\
\end{align*}
where $[\;]$ denotes the integer part of a real number. By the
last equality in above, $X_3$ if we know the value of $X_2$, the
value of $X_1$ will not give us extra information. Hence,
\[P(x_3|x_2,x_1)=P(x_3|x_2).\]
But since $[X_2]=[X_1]$, we also have
\[P(x_3|x_2,x_1)=P(x_3|x_1),\]
wherever the conditional probability is defined. This shows the
definition of neighbor is not well-defined in general.
\end{example}

Next we show that the positivity of the joint distribution implies
that the definition of neighbor is well-defined. By positivity of
the joint distribution, we mean
\[\forall x=(x_1,\cdots,x_n) \in
\Pi_{i=1}^n M_i,\; P(X_1=x_1,\cdots,X_n=x_n)>0.\]

\begin{lemma} Suppose $X_1,\cdots,X_n$ be a categorical random
field. If the joint distribution is strictly positive then the
concept of neighbor is well-defined for this field.
\end{lemma}
\begin{proof}
Suppose $\mc{J}=\{j_1,\cdots,j_J\}$ and
$\mf{H}=\{h_1,\cdots,h_H\}$ are sets of neighbors of site $i$.
Hence,
\begin{align*}
P(x_i|x_1,\cdots,x_{i-1},x_{i+1},\cdots,x_n)=f(x_{j_1},\cdots,x_{j_J})\\
\mbox{also},\\
P(x_i|x_1,\cdots,x_{i-1},x_{i+1},\cdots,x_n)=g(x_{h_1},\cdots,x_{h_H})\\
\end{align*}
For some functions $f,g$. By positivity condition, the conditional
probability is defined everywhere. Hence,

\[f(x_{j_1},\cdots,x_{j_J})=g(x_{h_1},\cdots,x_{h_H}),\;\forall x=(x_1,\cdots,x_n) \in \Pi_{i=1}^n
M_i.\] Suppose $h \in \mf{H}-\mf{J}$. Then $x_h$ does not appear
on the left hand side so $g$ is not dependent on $x_h$. We
conclude $\mf{H}-\mf{J}=\emptyset$. Similarly,
$\mf{J}-\mf{H}=\emptyset$.
\end{proof}

\section{Uninformative information sets}

In the following, we consider the general case (when the
positivity condition does not hold) and define some useful
concepts which are well-defined even though the concept of
neighbor is not as well-defined as defined by \cite{besag}.

We start by some useful definitions and lemmas regarding
conditional probabilities. Consider the conditional probability
$P(A|B)$ where $A,B$ are two events and $P(B)>0$. Also consider a
third event $C$. It is interesting to study when $C$ changes (or
does not change) our beliefs about probability of $A$. Formally,
we have the following definition.

\begin{definition} We call $C$ uninformative for $A$ given $B$ if
\[P(A|B,C)=P(A|B)\;\;\mbox{or}\;\; P(B,C)=0.\] Let $UN(A|B)$ to be the set of all events
$C$ such that $P(B,C)=0$ or $P(A|B,C)=P(A|B)$.
\end{definition}
\begin{lemma}
$UN(A|B)$ is closed under countable disjoint union.
\label{lemma-uninf-union}
\end{lemma}

\begin{proof}
Suppose, $\{C_i\}_{i=1}^{\infty}$ and $C_i \cap C_j=\emptyset,\; i
\neq j$. If for all $C_i$, $P(B \cap C_i)=0$ then result is trivial.
Otherwise, Let $I=\{i|\;P(B \cap C_i) \neq 0,\; i=1,2,\cdots\}.$

\begin{align*}
P(A|B, \cup_{i=1}^{\infty} C_i)=\frac{P(A,B, \cup_{i=1}^{\infty} C_i)}{P(B,\cup_{i=1}^{\infty} C_i)}=\\
\frac{\sum_{i \in I} P(A,B,C_i)}{\sum_{i \in I}P(B,C_i)}=\frac{\sum_{i \in I} P(A|B,C_i)P(B,C_i)}{\sum_{i \in I} P(B,C_i)}=\\
\frac{\sum_{i \in I} P(A|B)P(B,C_i)}{\sum_{i \in I} P(B,C_i)}=P(A|B).\\
\end{align*}

\end{proof}

One might also conjecture that $UN(A|B)$ is closed under
intersection. We show by some counterexamples, this is not true.

\begin{example}
$\Omega=\{1,2,3,4,5,6,7,8\},\, A=\{1,2,3,4\},\, B=\Omega,\,
C_1=\{2,4,6,8\},\, C_2=\{1,3,5,8\}$ and consider a uniform
probability distribution on $\Omega$.

Then $P(A|B)=P(A)=1/2,\;P(A|B,C_1)=P(A|B,C_2)=1/2$ hence $C_1,C_2
\in UN(A|B)$. But $P(A|B,C_1,C_2)=0$ while $P(B,C_1,C_2)=1/8 \neq
0.$
\end{example}

\begin{example}
Consider the joint distribution for $(X,Y,Z)$ given in Table 1,
where every row has the same probability of 1/4. Suppose that two
agents want to predict the value of $X$. The first person does not
have any information and the second one knows that $Z=0$. Now,
assume that we provide extra information to both agents. The extra
information is the value of $Y$. For the first agent at the
beginning (before the information about $Y$ was given):
$P(X=0)=P(X=1)=1/2$. After he knows the value of $Y$:
$P(X=1|Y=0)=P(X=1|Y=1)=1/2$. Hence, the extra information does not
change the belief of the first agent about $X$. One might conjecture
that since the second agent has more information than the first and
the new information did not help the first agent update his belief,
it should not change the belief of the second agent as well. This is
not true! In fact after getting the extra information, we have the
following inequality for the second agent:
\[0=P(X=1|Z=0,Y=1) \neq P(X=1|Z=0,Y=0)=1/2.\]

{\tiny
\begin{table}[H]
  \centering  \footnotesize
  \begin{tabular}{lccc}
\toprule[1pt]
X & Y & Z\\
\midrule[1pt]
  1 & 1 & 1\\
  1&0&0\\
  0&1&0\\
  0&0&0\\
  \bottomrule[1pt]
  \end{tabular}
  \caption{The joint distribution of $X,Y,Z$}
\end{table}}

\end{example}

We to prove a seemingly trivial fact about the conditional
probabilities in the following lemma.

\begin{lemma}
Suppose $P(A|B)$ is defined. Also suppose $\{C_i\}_{i=1}^{k},\,
k=1,2,\cdots,\infty$ a (finite or countable) collection of disjoint sets such that $\cup_{i=1}^k C_i=\Omega$. Assume
\[P(B,C_i)=0\;\mbox{or}\;P(A|B,C_i)=c.\]
In other words, $P(B,C_i)$ does not depend on $C_i$. Then $C_i \in
UN(A|B)$:
\[P(A|B,C_i)=P(A|B)\;\mbox{or}\;P(B,C_i)=0.\]
\label{lemma-uninf-charac}
\end{lemma}
\begin{proof}
Let $I=\{i|\;1\leq i \leq k,\;P(B,C_i)>0\}$. Then we have
\[P(A|B)=\frac{\sum_{i=1}^{k}P(A,B,C_i)}{\sum_{i=1}^{k}P(B,C_i)}=\]
\[\frac{\sum_{i \in I}P(A,B,C_i)}{\sum_{i \in I}P(B,C_i)}=\]
\[\frac{\sum_{i \in I}P(A|B,C_i)P(B,C_i)}{\sum_{i \in I}P(B,C_i)}=\]
\[\frac{\sum_{i \in I}c P(B,C_i)}{\sum_{i \in I}P(B,C_i)}=c.\]
\end{proof}

\begin{corollary}
Suppose $P(x_i|x_{i_1},\cdots,x_{i_I})$ depends only on
$x_{j_1},\cdots,x_{j_J}$, where \[\{j_1,\cdots,j_J\} \subset
\{i_1,\cdots,i_I\},\] when the conditional probability,
$P(x_i|x_{i_1},\cdots,x_{i_I})$ is defined. Then
\[P(x_i|x_{i_1},\cdots,x_{i_I})=P(x_i|x_{j_1},\cdots,x_{j_J}),\] when
the conditional probability, $P(x_i|x_{i_1},\cdots,x_{i_I})$ is
defined.
\end{corollary}
\begin{proof} Fix $(x_{j_1}',\cdots,x_{j_J}')$. Let $A=\{X_i=x_i\}$
and $B=\{X_{j_1}=x_{j_1}',\cdots,X_{j_J}=x_{j_J}'\}.$ Let \[\{k_1,\cdots,k_K\}=\{i_1,\cdots,i_I\}-\{j_1,\cdots,j_J\}.\] Consider the sets
\[C_{x_{k_1},\cdots,x_{k_K}}=\{X_{k_1}=x_{k_1},\cdots,X_{k_K}=x_{k_K}\},\;\;x_{k_l}\in M_{k_l}.\] 
These sets are disjoint, there exist finitely many of them and their union is $\Omega$. Then by the assumption
$P(A|B,C_{x_{k_1},\cdots,x_{k_K}})=c,\;$ or $P(B,C_{x_{k_1},\cdots,x_{k_K}})=0.$ Now apply Lemma
\ref{lemma-uninf-charac} to $A,B,C_{x_{k_1},\cdots,x_{k_K}}$.
\end{proof}

\section{Sufficient and minimal information sets}

This section introduces minimal and sufficient information
sets. Suppose we have $n$ sites in the random field indexed by
$1,2,\cdots,n$. We denote a site by $i$. Let
$i^c=\{1,2,\cdots,n\}-\{i\}$ be the set of all other sites other
than site $i$. Let $\mc{I}=\{i_1,\cdots,i_I\}\subset
\{1,2,\cdots,n\}$ be a collection of sites and let
\[D_{\mf{I}}=D_{i_1,\cdots,i_I}=\{(x_{i_1},\cdots,x_{i_I})|P(x_{i_1},\cdots,x_{i_I})>0 \}\]
Note that $D$ depends on the set of the subscripts and not the order
of them. Also note that $D$ is the domain where the conditional
probability given the values on the sites $\mc{I}$ is defined. By
$p(i|\mc{I})$, we mean the conditional probability of site $i$ given
$\mc{I}$ defined on $E_{i;{\mc{I}}}=M_i\times D_{\mc{I}}$. Also note
that with the positivity of the joints distributions assumption:
\[D_{\mc{I}}=D_{i_1,\cdots,i_I}=\Pi_{j=1}^I M_{i_j}.\]
Since the concept of neighbor is not well-defined in the general
case, we seek other useful definitions to study the general case.

Note that $P(i|\mc{I})$ is a function
\[P(i|\mc{I}):M_i \times D_{\mc{I}} \rightarrow [0,1],\]
\[P(x_i|x_{i_1},\cdots,x_{i_I})=P(X_i=x_i|X_{i_1}=x_{i_1},\cdots,X_{i_I}=x_{i_I}).\]

\begin{definition} Sufficient information set: Suppose $\mc{J} \subset \mc{I} \subset
\{1,2,\cdots,n\}$, $\mc{J}$ is called a sufficient information set
for $i$, given $\mc{I}$, if
\[P(i|\mf{I})=P(i|\mf{J}),\]
on $E_{i;\mc{I}}$. We denote the set of all such sets by
$SI(i,\mc{I})$.
\end{definition}

\begin{definition}
$\mc{I}\subset {1,2,\cdots,n}$ is called a minimal information set
for $i$ if $P(i|\mc{I})\neq P(i| \mc{J})$ for any
$\mc{J},\;\mc{J}\subset \mc{I},\mc{J} \neq \mc{I}$. We denote the
set of all such sets by $MI(i)$.
\end{definition}

In the following, we study the properties of $SI$ (sufficient
information) and $MI$ (minimal information) sets.

First, let us see what happens if $i \in \mc{I}$. In this case,
$\{i\} \in SI(i,\mc{I})$. Also, note that in general $\{i\} \in
MI(i)$ if $m_i>1$.  (If $m_i=1$ then we need no information to say
what the value of site $i$ is.) Also note that $\emptyset \in MI(i)$
in general.

One might conjecture a smaller a set than a given minimal
information set is a minimal set as well. This is not true! In
example 3,  $\{Y,Z\} \in MI(X)$ but $\{Y\}$ is not minimal since
$P(X|Y)=P(X|\emptyset)$.

\begin{proposition}
Suppose $\mc{J} \in SI(i,\mc{I})$ and $\mc{H}=\mc{I}-\mc{J}$. Also
assume
\[\emptyset \neq N_{h_1} \subset M_{h_1},\cdots, \emptyset \neq N_{h_H} \subset M_{h_H}\]
then
\[P(i|\mc{J})=P(i|\mc{J},x_{h_1}\in N_{h_1},\cdots,x_{h_H}\in N_{h_H}),\]
whenever, the right hand side is defined.
\end{proposition}

\begin{proof}
Fix $(x_{j_1}',\cdots,x_{j_J}')$, we want to show

\[P(x_i|x_{j_1}',\cdots,x_{j_J}',x_{h_1}\in
N_{h_1},\cdots,x_{h_H}\in
N_{h_H})=P(x_i|x_{j_1}',\cdots,x_{j_J}'),\] whenever the left hand
side is defined. But
\[P(x_i|x_{j_1}',\cdots,x_{j_J}',x_{h_1},\cdots,x_{h_H})=P(x_i|x_{j_1}',\cdots,x_{j_J}'),\]
or \[P(x_{j_1}',\cdots,x_{j_J}',x_{h_1},\cdots,x_{h_H})=0,\] since
$\mc{J}$ is sufficient. Now use the fact that $UN$ is closed under
disjoint union and take the union over

$\{X_{j_1}=x_{j_1}',\cdots,X_{j_J}=x_{j_J}',X_{h_1}=x_{h_1},\cdots,X_{h_1}=x_{h_H}\}_{x_{h_1}\in
N_{h_1},\cdots,x_{h_H}\in N_{h_H}}$

\end{proof}

\begin{lemma}
a) If $\mc{J} \in SI(i,\mc{I})$ and $\mc{J} \subset \mc{H} \subset
\mc{I}$ then $\mc{J} \in SI(i,\mc{H})$.\\
b) If $\mc{J} \in SI(i,\mc{I})$ and $\mc{J} \subset \mc{H} \subset
\mc{I}$ then $\mc{H} \in SI(i,\mc{I})$.\\
\end{lemma}

\begin{proof}

Let $\mc{K}=\mc{I}-\mc{H}$. $\mc{K}=\{k_1,\cdots,k_K\}.$ We want to
show that for a fixed $(x_{i_1}',\cdots, x_{i_I}')\in D_{\mc{I}},$

\[{\rm a)}\; P(x_i|x_{h_1}',\cdots, x_{h_H}')=P(x_i|x_{j_1}',\cdots, x_{j_J}'),\]
\[{\rm b)}\; P(x_i|x_{i_1}',\cdots, x_{i_I}')=P(x_i|x_{h_1}',\cdots, x_{h_H}')\]
By assumption for all $(x_{i_1},\cdots, x_{i_I})$ which their
restriction to indices in $K$ is\\  $(x_{k_K}',\cdots, x_{k_K}')$
either $P(x_{i_1},\cdots, x_{i_I})=0$ or
\[P(x_i|x_{i_1},\cdots, x_{i_I})=P(x_i|x_{j_1}',\cdots, x_{j_J}').\]
On the left hand side take the union over
$\{X_{k_1}=x_{k_1},\cdots,X_{k_K}=x_{k_K}\}_{x_{k_l}\in M_{k_l}}.$
We get

\[P(x_i|x_{h_1}',\cdots, x_{h_H}')=P(x_i|x_{j_1}',\cdots, x_{j_J}')=P(x_i|x_{i_1}',\cdots, x_{i_I}').\]
\end{proof}

To generalize the concept of neighbor, we can use the sufficient
information and minimal information sets. We call a set efficiently
sufficient for site $i$ if it is minimal and sufficient for $i$
given $i^c$. i.e. $\mc{I}$ is efficiently sufficient for $i$ if and
only if $\mc{I} \in MI(i) \cap SI(i,i^c)$. We denote the set of all
such sets $ES(i)$. If for some $i$, $ES(i)$ has only one element, we
call that element a neighbor of site $i$. Note that the definition
of neighbor coincide with the definition of neighbor by \cite{besag} and \cite{cressie} if the positivity condition
holds. In the following example we show that this is not necessary.

\begin{example}
Consider the joint distribution of $X,Y$ as given by Table 2, where
every row is equally probable. Then the positivity condition does
not hold since $P(X=1,Y=0)=0$. But for $X$, the site $Y$ is a neighbor since
$Y \in MI(X) \cap SI(X,Y)$. Also for $Y,$ $X$ is a neighbor. {\tiny
\begin{table}[H]
  \centering  \footnotesize
  \begin{tabular}{lcc}
\toprule[1pt]
X & Y \\
\midrule[1pt]
  1 & 1 \\
  0&1\\
  0&0\\
  \bottomrule[1pt]
  \end{tabular}
  \caption{The joint distribution of $X,Y$}
\end{table}}

\end{example}

\bibliographystyle{plainnat}
\bibliography{D://Research//mybibreza}
\end{document}